\documentclass[a4paper, reqno, 14pt]{amsart}
\usepackage[utf8]{inputenc}

\usepackage{amsthm,amsfonts,amssymb,amsmath, microtype}
\usepackage[all]{xy}
\usepackage{xr-hyper}
\usepackage[colorlinks=true]{hyperref}
\usepackage{verbatim}
\usepackage{pdfsync}

\newcommand{\BF}{{\mathbb {F}}}

\newcommand{\BN}{{\mathbb {N}}}

\newcommand{\BQ}{{\mathbb {Q}}}

\newcommand{\BV}{{\mathbb {V}}}

\newcommand{\BX}{{\mathbb {X}}}
\newcommand{\BY}{{\mathbb {Y}}}
\newcommand{\BZ}{{\mathbb {Z}}}

\newcommand{\CD}{{\mathcal {D}}}

\newcommand{\CN}{{\mathcal {N}}}
\newcommand{\CO}{{\mathcal {O}}}

\newcommand{\CZ}{{\mathcal {Z}}}

\newcommand{\End}{{\mathrm{End}}}

\newcommand{\charpol}{{\mathrm{charpol}}}

\newcommand{\Hom}{{\mathrm{Hom}}}

\newcommand{\Lie}{{\mathrm{Lie}}}

\newcommand{\Spec}{{\mathrm{Spec}}}

\newcommand{\id}{{\mathrm{id}}}

\newtheorem{theorem}{Theorem}

\newtheorem{lemma}[theorem]{Lemma}

\theoremstyle{definition}

\newtheorem{remark}[theorem]{Remark}

\begin{document}

\title{On the regularity of special difference divisors}
\author{Ulrich Terstiege}

\date{\today}
\maketitle
In the paper \cite{T} a close connection between intersection multiplicities of special cycles on the Shimura variety for $GU(1,2)$ and Fourier coefficients of the derivative of a certain Eisenstein series for $U(3,3)$ is established. This confirms in the case $n=3$ a conjecture of Kudla and Rapoport for the Shimura variety for $GU(1,n-1)$, see \cite{KR} and \cite{KR2}. This conjecture can be reduced to a local statement on intersection multiplicities on suitable Rapoport-Zink spaces stated in \cite{KR}. The case of non-degenerate (i.e. $0$-dimensional) intersections which can be reduced to the case $n=2$ is proved in \cite{KR} and \cite{KR2}. For $n\geq 3$, the intersection of the special cycles in general has positive dimension. To a special cycle one associates a difference divisor, see below. A crucial ingredient of the proof of the Kudla-Rapoport conjecture for $n=3$ is the regularity of these difference divisors. In this note we prove the regularity of special difference divisors in arbitrary dimension. It is hoped that this will help to prove the Kudla-Rapoport conjecture for arbitrary $n$.

 Let us recall the setting. 
Let $n$ be a positive integer and let $p\geq 3$ be a prime. Let $\BF=\overline\BF_p$ and let $W=W(\BF)$ be its ring of Witt-vectors. Let  $\CN:=\CN_n:=\CN(1,n-1)$ be the Rapoport-Zink space over $W$ parameterizing  tuples $(X,\iota,\lambda,\rho)$ over $W$-schemes $S$ such that $p$ is locally nilpotent in $\CO_S$. Here a tuple $(X,\iota,\lambda,\rho)$ over $S$ consists of the following objects. First, $X$ is a $p$-divisible group of dimension $n$ and height $2n$ over $S$, and  $\iota: \BZ_{p^2}\rightarrow \End(X)$ is a homomorphism satisfying the determinant condition of signature $(1,n-1)$, i.e. $$
\charpol (\iota(a), \Lie X)(T)=(T-\phi_0(a))(T-\phi_1(a))^{n-1} \in \mathcal{O}_S [ T ], 
$$ where $\phi_0$ and $\phi_1$ are the two embeddings of $\mathbb Z_{p^2}$ into $W$.  Further $\lambda$ is a principal polarization of $X$ such that for the Rosati involution we have $\iota^*(a)=\iota(\overline a)$ for all $a\in \BZ_{p^2}$, and $$\rho: X\times_S \overline S\rightarrow \BX\times_{\Spec \  \BF} \overline S$$ is a $\BZ_{p^2}$-linear quasi-isogeny of height $0$. Here $\overline S= S\times_{\Spec \ W} \Spec \ \BF$ and $(\BX,\iota_{\BX},\lambda_{\BX})$ is a fixed triple over $\Spec \ \BF$ as before and where $\BX$ is also required to be supersingular. We also require that locally
up to a scalar in $\BZ_p^{\times}$ we have the identity ${\rho}^{\vee}\circ\lambda_{\mathbb{X}}\circ \rho = \lambda$.
Let $(\BY,\iota_{\BY},\lambda_{\BY})$ over $\BF$ be the fixed supersingular object for $n=1$ and let $\overline \BY$ be same object but the $\BZ_{p^2}$-action replaced by its conjugate. It has a canonical lift $\overline Y$ over $W$, cf. \cite{G}.   The space of special homomorphisms is defined  as $\BV=\Hom_{\BZ_{p^2}}(\overline \BY,\BX)\otimes \BQ$, cf. \cite{KR}. It is an $n$-dimensional hermitian $\BQ_{p^2}$-vector space with hermitian form $h$ given by $$h(x,y)= \lambda^{-1}_{\overline{\BY}}\circ {y^{\vee}}\circ \lambda_{\BX}\circ x \in \End_{\BZ_{p^2}}(\overline{\BY})\otimes \BQ \cong \BQ_{p^2},$$
where the last isomorphism is via $\iota_{\overline \BY}^{-1}$, and where $y^{\vee}$ denotes the dual of $y$. By the valuation of a special homomorphism $j$ we mean the $p$-adic valuation of $h(j,j)$. Given $j_1,...,j_m\in\BV$,  the fundamental matrix $T(j_1,...,j_m)$ of $j_1,...,j_m$ is the hermitian $m\times m$ matrix with entry $h(j_i,j_k)$ at $i,k$. For $j\in \BV$ the special cycle $\CZ(j)$ is the closed formal subscheme of $\CN$ such that $\CZ(S)$ is the set of all $(X,\iota,\lambda,\rho)$ over $S$ such that the quasi-homomorphism $$\overline{\BY}\times_{\BF} \overline{S}\stackrel{j}{\longrightarrow}{\BX}\times_{\BF} \overline{S}\stackrel{\varrho^{-1}}{\longrightarrow}X\times_S \overline{S} $$ lifts to a homomorphism $\overline Y\times_{\Spec \ W} S \rightarrow X$. If the valuation of $j\neq 0$ is non-negative then $\CZ(j)$ is a relative divisor, i.e. it is flat over $W$; if the valuation of $j$ is negative then $\CZ(j)$ is empty.  Further we define  the special difference divisor $\CD(j)$ as  $\CD(j)=\CZ(j)-\CZ(j/p)$. Thus, if a local equation for $\CZ(j)$ is given by $f=0$ and a local equation for $\CZ(j/p)$ is given by $g=0$ then a local equation for $\CD(j)$ is given by $fg^{-1}=0$. Note here that $g$ divides $f$ since obviously $\CZ(j/p)\subseteq \CZ(j)$. See \cite{KR} and \cite{T} for more information about these notions.

 Recall from \cite{RTZ} that an $\BF$-valued point $x$ of $\CN$ is called {\emph{super-general}} if there is no special homomorphism $j$ of valuation $0$ such that $x\in \CZ(j)(\BF)$. In \cite{RTZ}, Theorem 10.7, it is proved that if $x\in \CN(\BF)$ is a super-general point and if $j\in \BV$ is such that $x\in \CZ(j)(\BF)$ but $x\not\in \CZ(j/p)(\BF)$, then the special fiber $\CZ(j)_p$ of $\CZ(j)$ is regular at $x$.

\begin{theorem}\label{mainth}
Let $j$ be a special homomorphism. Then the special difference divisor $\CD(j)$ is regular. 
\end{theorem}
\begin{proof}
We proceed by induction on $n$ and may assume that $j\neq 0 $. The claim is known for $n\leq 3$, comp. \cite{T}. Suppose the claim is true for $n-1\geq 3$.

The first part of the following lemma is (in a local version) already implicitly contained in \S 5 of \cite{KR}.

\begin{lemma}\label{bw0}
i) Let $j\in\BV$ a special homomorphism of valuation $0$. Then $\CZ(j)$ is isomorphic to $\CN_{n-1}$. 

\noindent ii) If $j_1,...,j_m\in\BV$ and $h(j,j_i)=0$ for all $i$ (where $j$ is as in i)) then  there are special homomorphisms $j_1',...,j_m'$ for  $\CN_{n-1}$ such that $T(j_1,...,j_m)=T(j_1',...,j_m')$ and such that under the isomorphism of i) $\CZ(j)\cap \CZ(j_1)\cap...\cap \CZ(j_m)$ is identified with $\CZ(j_1')\cap...\cap \CZ(j_m')$.
\end{lemma}
\begin{proof}
By multiplying $j$ by a suitable element in $\BZ_{p^2}^{\times}$ we may assume that $h(j,j)=1$. For any tuple $(X,\iota,\lambda,\rho)$ we may replace the polarization $\lambda$ by a suitable multiple of $\lambda$ such   that  ${\rho}^{\vee}\circ\lambda_{\mathbb{X}}\circ \rho = \lambda$. 
We may also assume that $(\BX, \iota_{\BX}, \lambda_{\BX},\id)\in\CZ(j)(\BF)$. 
Let $(X,\iota,\lambda,\rho)\in\CZ(j)(S)$. Then $e=\rho^{-1}j(\rho^{-1}j)^*$ defines an idempotent of $X$, where we write here $(\rho^{-1}j)^*=\lambda_{\overline\BY}^{-1}\circ (\rho^{-1}j)^{\vee}\circ \lambda$.  Then  we can decompose $X$ as $X=eX\times (1_X-e)X$. The map $(X,\iota,\lambda,\rho)\mapsto ((1_X-e)X,\iota_{|(1_X-e)X}, (1_X-e)^{\vee}\circ\lambda_{| (1_X-e)X},(1_\BX-jj^*)\circ\rho_{| (1_X-e)X})$ then defines an isomorphism $\CZ(j)\rightarrow \CN_{n-1}$. Its inverse is given by sending $(X,\iota,\lambda,\rho)\in \CN_{n-1}(S)$ to $(X\times \overline Y_S\,\iota\times \iota_{\overline Y_S},\lambda\times \lambda_{\overline Y_S},\rho\times \id)\in \CZ(j)(S)\subseteq \CN_{n}(S)$ (here we have identified $\overline Y_S$ with $j \overline Y_S$). One easily checks that these maps are inverse to each other (up to isomorphism).
 
 Any $j_i$ induces an element $j_i'=(1_\BX-jj^*)\circ j_i\in \Hom_{\BZ_{p^2}}(\overline\BY, (1_\BX-jj^*)\BX)\otimes \BQ$ and using the fact that $h(j,j)=1$ and $h(j,j_i)=0$ one easily checks that $T(j_1,...,j_m)=T(j_1',...,j_m')$ and by construction the intersection $\CZ(j)\cap \CZ(j_1)\cap...\cap \CZ(j_m)$ is identified with $\CZ(j_1')\cap...\cap \CZ(j_m')$.
\end{proof}

\begin{lemma}\label{elemreg}
Let $x\in \CN(\BF)$. Let $j\in \BV$ and suppose that  $x \in \CZ(j)(\BF)$ but $x \not\in \CZ(j/p)(\BF)$. Then $\CD(j)$ is regular at $x$. If $x$ is super-general  or if $j$ is of valuation $0$, then also the special fiber $\CD(j)_p$ is regular at $x$. 
\end{lemma}

\begin{proof}
If $x$ is super-general, then the claim follows from Theorem 10.7 in \cite{RTZ}. If $j$ is of valuation $0$, then we know by Lemma \ref{bw0} that $\CD(j)=\CZ(j)\cong \CN_{n-1}$ which is regular and also its special fiber is regular.
Now we proceed by induction on $n$, the induction start is given by the induction start of the theorem. 
Suppose the claim is true for $n-1\geq 3$.
We may assume that  $x$ is not super-general and that  $j$ is not of valuation $0$. It follows that there is a special homomorphism $j_0$ of valuation $0$ which is linearly independent of $j$ and such that $x\in \CD(j_0)(\BF).$ We write $j=\alpha j_0+ \beta j_1$, where $j_1\perp j_0$ and $x\in \CZ(j_1)(\BF)$ but $x\not\in \CZ(j_1/p)(\BF)$ and $\alpha,\beta\in \BZ_{p^2}$. We claim that $\beta $ is not divisible by $p$. For, if $\beta$ and $\alpha$ are divisible by $p$, then $x\in \CZ(j/p)(\BF)$ in contradiction to our hypothesis and if $\beta$ is divisible by $p$ but $\alpha$ is not divisible by $p$, then it follows that the valuation of $j$ is $0$, again a contradiction to our hypothesis. Thus $\beta\in \BZ_{p^2}^{\times}$ and it follows  that locally around $x$ we have $\CD(j)\cap \CZ(j_0)=\CZ(j)\cap \CZ(j_0)=\CZ(j_1)\cap \CZ(j_0)$ and this is  regular at $x$ by the induction hypothesis of the theorem, since $\CZ(j_1)\cap \CZ(j_0)=\CD(j_1)\cap \CZ(j_0)$ (locally around $x$) which can be viewed as a difference divisor in $\CZ(j_0)\cong \CN_{n-1}$ by Lemma \ref{bw0}. 
\end{proof}

\begin{lemma}\label{pmult}
Let  $j \neq 0$ be a special homomorphism of non-negative valuation,  
let
  $x \in \CZ(j)(\BF)$, and let $D\subset {\rm Spf }(\CO_{\CN,x}) $ be  a regular divisor. Suppose that  in $ {\rm Spf }(\CO_{\CN,x})$  we have for the special fibers the inclusion  $D_p \subset \CZ (j)_p$. 
  Let   $(f)$ be the ideal of $\CZ(j)\cap D$ in $\CO_{D,x}$ (i.e., $\CO_{\CZ(j)\cap D,x}=\CO_{D,x}/(f)$).
   Then  the ideal of $\CZ (pj)\cap D$ in $\CO_{D,x}$ is  $( p \cdot f).$ Thus, if $f \neq 0$, the equation of  $\CD(pj) \cap D$ in $\CO_{D,x}$ is given by $p=0.$
\end{lemma}
\noindent This is proved in the same way as Lemma 2.11. in \cite{T}. \qed
\vspace{0.1cm}

We proceed with proving the theorem.
Let $x\in \CZ(j)(\BF)$. Define $\gamma\in \BN$ by the property that $x\in \CZ(j/p^{\gamma})(\BF)$ but $x \not\in \CZ(j/p^{\gamma+1})(\BF)$. Writing $j=p^{\gamma}j'$,  
 it is enough to prove the following claim by induction on $a$.
 \smallskip

{\emph{Claim. For any $a \in \BN$ the  special difference divisor $\CD(p^aj')$ is regular at $x$. }}

 \smallskip
\noindent The induction start (i.e. $a=0$) is given by Lemma \ref{elemreg}. 
Suppose now the claim is true for $a-1$.
We distinguish the cases that $\CD(j')_p$ is not regular at $x$ and that $\CD(j')_p$ is regular at $x$. 

Suppose first that $\CD(j')_p$ is not regular at $x$. Then by Lemma \ref{elemreg} $x$ is not super-general and $j'$ is not of valuation $0$ and thus there is a special homomorphism $j_0$ of valuation $0$ such that $x\in \CZ(j_0)(\BF)$ and such that $j'$ and $j_0$ are linearly independent. Thus as in the proof of Lemma \ref{elemreg} we may write $j'=\alpha j_0+ \beta j_1$, where $j_1\perp j_0$ and $x\in \CZ(j_1)(\BF)$ but $x\not\in \CZ(j_1/p)(\BF)$ and $\alpha \in  \BZ_{p^2}$ and $\beta\in \BZ_{p^2}^{\times}$.  Then it follows that $\CD(j_0)\cap \CD(p^aj')=\CD(j_0)\cap \CD(p^a\beta j_1)$ locally around $x$. This is regular at $x$ by the induction hypothesis (on $n-1$) since it can be viewed as a difference divisor in $\CD(j_0)=\CZ(j_0)\cong \CN_{n-1}$ locally around $x$ (Lemma \ref{bw0}). Thus $\CD(p^aj')$ is also  regular at $x$.

Suppose now that $\CD(j')_p$ is regular at $x$. 
Let  $f_i=0$ be the equation of $\CD(p^i\cdot j')$ in $R=\CO_{\CN,x}$ which is a UFD. Then by the induction hypothesis (on $a-1$) we know that all $f_i$ for $i\leq a-1$ are prime elements. 
Next we show that there is a unit $\eta \in R^{\times}$ such that $f_0+\eta p$ is 
 coprime to $f_i$ for all  $i\leq a-1$. For $i=0$ this is true for any unit $\eta$  since $\CZ(j')$ is a relative divisor. Note also  that $f_0+\eta p$ is a prime element for any $\eta$ since  $\CZ(j')_p$ is regular at $x$ and hence also $V(f_0+\eta p)$ is a regular divisor in ${\rm Spf} (R)$ .  For any $z\in R$ denote by $\overline{z}$ its image in $\overline{R}:= R/(p)$.
 Now we claim that if $\overline{f_i}$ and $\overline{f_0}$ do not differ by only a unit, then $f_0+\eta p$ is coprime to $f_i$ for any $\eta$. To see this note that $f_0+\eta p$  and $ f_i$ are prime, hence if they  are not coprime they only differ by a unit, hence in this case also $\overline{f_i}$ and $\overline{f_0}$ only differ by a unit in $\overline{R}$.
 
 Suppose that $i_1,...,i_l$ are the indices such that $\overline{f_{i_1}},...,\overline{f_{i_l}} $ each differs from $\overline{f_0}$ by a unit. This means that after perhaps multiplying the $f_{i_k}$ by suitable units we have $f_{i_k}=f_0+py_{i_k}$ for suitable $y_{i_k}\in R$. Choose now $y\in R$ such that modulo $\mathfrak m$ (the maximal ideal in $R$) the relations 
 $$y\not\equiv -(y_{i_1}+\dotsb + y_{i_l}),\ y\not\equiv y_{i_1}-(y_{i_1}+\dotsb + y_{i_l}), \ \hdots, \ y\not\equiv y_{i_l}-(y_{i_1}+\dotsb + y_{i_l})$$ hold. (This is possible since $\BF$ is an infinite field.) Then it follows that $\eta := y_{i_1}+\dotsb + y_{i_l}+y$ is a unit in $R$. Furthermore $f_{i_k}-(f_0+\eta p)=f_0+py_{i_k}-(f_0+\eta p)=p(y_{i_k}-(y_{i_1}+\dotsb + y_{i_l}+y))$. Since by the above relations $y_{i_k}-(y_{i_1}+\dotsb + y_{i_l}+y) $ is a unit and since $f_{i_k}$ is coprime to $p$, it follows that indeed $f_0+\eta p$ is 
 coprime to  $f_{i_k}$ and thus to $f_i$ for any $i\leq a-1$.
 
 Denote now by $D$ the divisor in ${\rm Spf} (R)$ given by $f_0+\eta p=0$. Then since $\CD(j')_p$ is by  assumption regular at $x$, the special fiber $D_p$ is regular and hence $D$ also is regular. Further $D_p\subseteq \CZ(p^{a-1}j')_p$ in ${\rm Spf} (R)$ and the equation of $\CZ(p^{a-1}j')$ in  $\CO_{D,x}$ is not zero since $f_i$ and $f_0+p\eta$ are coprime for any $i\leq a-1$.
Thus by  Lemma \ref{pmult} the ideal of $D\cap \CD(p^aj')$ in  ${\rm Spf} (R)$ is $(f_0+\eta p,p)=(f_0,p)$, thus $D\cap \CD(p^aj')=D_p$ which is regular (as a closed formal subscheme of ${\rm Spf} (R)$), hence $\CD(p^aj')$ is regular at $x$. This ends the proof.
\end{proof}

\begin{remark}
Using displays for $p$-divisible groups one can show that for $i \geq 1$ the special fiber of $V(f_i)$ in ${\rm Spf} (R)$ is not regular  and hence not equal to the special fiber of $V(f_0)$, hence one can choose $\eta=1$ in the above proof. This reasoning also shows  that in general the intersection of several (even of two) difference divisors is not regular.  
\end{remark}

\subsection*{Acknowledgements}
I thank U. G\"ortz and M. Rapoport  for helpful remarks on the text.


\begin{thebibliography}{AB3}

\bibitem{G} B. Gross, \textit{On canonical and quasi-canonical liftings}, Invent. math. {\bf 84} (1986), 321--326. 

\bibitem{KR} S. Kudla, M. Rapoport, \textit{Special cycles on unitary Shimura varieties, I. Unramified local theory}, Invent. math. {\bf 184} (2011), 629--682. 

\bibitem{KR2} S. Kudla, M. Rapoport, \textit{Special cycles on unitary Shimura varieties, II. Global theory},  arXiv:0912.3758v1

\bibitem{RTZ}  M. Rapoport, U. Terstiege, W. Zhang, \textit{On the Arithmetic Fundamental Lemma in the minuscule case}, http://arxiv.org/abs/1203.5827

\bibitem{T} U. Terstiege, {\em Intersections of special cycles on the Shimura variety for $GU(1,2)$},   to appear in J. reine angew. Math. 

\end{thebibliography}
\end{document}